\documentclass{amsart}

\usepackage[latin1]{inputenc}
\usepackage[T1]{fontenc}
\usepackage{enumerate}
\usepackage{mathtools}
\usepackage{amssymb}
\usepackage{mathrsfs}
\usepackage{mathabx}

\newtheorem{theorem}{Theorem}[section]
\newtheorem{lemma}[theorem]{Lemma}
\newtheorem{cor}[theorem]{Corollary}

\theoremstyle{definition}

\theoremstyle{remark}
\newtheorem{remark}[theorem]{Remark}

\numberwithin{equation}{section}

\newcommand{\K}{\mathbb K}

\title[Invariant subspaces for Fréchet spaces without continuous norm]{Invariant subspaces for Fréchet spaces without continuous norm}
\author[Q. Menet]{Quentin Menet}

\address{Quentin Menet, Département de Mathématique, Université de Mons, 20 Place du Parc, 7000 Mons, Belgique}
\email{quentin.menet@umons.ac.be}
\thanks{Quentin Menet is a Research Associate of the Fonds de la
Recherche Scientifique - FNRS and was supported by the grant ANR-17-CE40-0021 of the French National 
Research Agency ANR (project Front)}

\allowdisplaybreaks[3]
\begin{document}
\begin{abstract}
Let $(X,(p_j))$ be a Fréchet space with a Schauder basis and without continuous norm, where $(p_j)$ is an increasing sequence of seminorms inducing the topology of $X$.
We show that $X$ satisfies the Invariant Subspace Property if and only if there exists $j_0\ge 1$ such that $\ker p_{j+1}$ is of finite codimension in $\ker p_{j}$ for every $j\ge j_0$. 
\end{abstract}
\maketitle
\section{Introduction}

A Fréchet space $X$ satisfies the Invariant Subspace Property if every operator $T\in L(X)$ possesses a non-trivial invariant subspace, i.e. a closed subspace $M$ different from $\{0\}$ and $X$ such that $TM\subset M$. It is an important question in functional analysis to determine if a space $X$ satisfies this property. Thanks to Enflo and Read, we know that there exist separable infinite-dimensional Banach spaces that do not satisfy the Invariant Subspace Property~\cite{Enflo1, Enflo2, Read1}. For instance, the spaces $\ell_1(\mathbb{N})$ and $c_0(\mathbb{N})$ do not satisfy this property~\cite{Read2, Read3}. On the other hand, there exist separable infinite-dimensional Banach spaces that satisfy the Invariant Subspace Property~\cite{Argyros1}. However, a characterization of Banach spaces satisfying the Invariant Subspace Property is still far from being found. In particular, it is not known if separable infinite-dimensional Hilbert spaces satisfy the Invariant Subspace Property and this problem is called the Invariant Subspace Problem.

Fréchet spaces are a natural generalization of Banach spaces where the topology is not necessarily given by one norm but can be given by an increasing sequence of norms or, more generally, by an increasing sequence of seminorms. These spaces can be divided into three families:
\begin{itemize}
\item Banach spaces;
\item Non-normable Fréchet spaces with continuous norm;
\item Fréchet spaces without continuous norm.
\end{itemize}
Moreover, Fréchet spaces without continuous norm $(X,(p_j))$ can still be divided into two subfamilies regarding the codimension of $\ker p_{j+1}$ in $\ker p_{j}$.

We can therefore wonder which Fréchet spaces satisfy the Invariant Subspace Property. Both behaviors can also be found among (non-normable) Fréchet spaces. There exist separable infinite-dimensional non-normable Fréchet spaces that satisfy the Invariant Subspace Property~\cite{Korber,Shields} and a classical example is given by $\omega$, which is the space of all complex sequences endowed with the seminorms $p_j(x)=\max_{n\le j}|x_n|$. On the other hand, there exist separable infinite-dimensional non-normable Fréchet spaces that do not satisfy the Invariant Subspace Property~\cite{Atzmon, Atzmon2}. Classical examples of such spaces are the space of entire functions $H(\mathbb{C})$ and the Schwartz space of rapidly decreasing functions $s$~\cite{Golinski1, GolinskiPhd, Golinski2}.

A first study trying to determine which non-normable Fréchet spaces satisfy the Invariant Subspace Property was recently carried out~\cite{Menet}. For separable infinite-dimensional non-normable Fréchet spaces with continuous norm, it was shown that a lot of them do not satisfy the Invariant Subspace Property and for Fréchet spaces without continuous norm $(X,(p_j))$, it was proved that if for every $j\ge j_0$, $\ker p_{j+1}$ is a subspace of finite codimension in $\ker p_j$ then $X$ satisfies the Invariant Subspace Property. This last result means that if $Y$ is a Fréchet space with a continuous norm then $\omega\oplus Y$ satisfies the Invariant Subspace Property. In particular, the spaces $\omega\oplus\ell_1(\mathbb{N})$ and $\omega\oplus\ell_2(\mathbb{N})$ satisfy the Invariant Subspace Property. However, we could not determine if a Fréchet space without continuous norm that do not meet the above condition satisfies the Invariant Subspace Property. The goal of this paper consists in investigating these spaces in order to characterize Fréchet spaces without continuous norm that satisfy the Invariant Subspace Property.

A first result concerning this subfamily of Fréchet spaces without continuous norm was obtained by Goli\'{n}ski and Przestacki~\cite{Golinski3}. Indeed, they proved that the space of smooth functions on the real line does not satisfy the Invariant Subspace Property. We will use several key ideas of this paper to show the following result.

\begin{theorem}\label{technique}
Let $(X,(p_j))$ be a Fréchet space with a Schauder basis and without continuous norm, where $(p_j)$ is an increasing sequence of seminorms inducing the topology of $X$.
The space $X$ satisfies the Invariant Subspace Property if and only if there exists $j_0\ge 1$ such that $\ker p_{j+1}$ is of finite codimension in $\ker p_{j}$ for every $j\ge j_0$. 
\end{theorem}

In view of Theorem 2.1 in~\cite{Menet}, one of the two implications is already known, and we only need to prove that if for every $j_0\ge 1$, there exists $j\ge j_0$ such that $\ker p_{j+1}$ is of infinite codimension in $\ker p_{j}$ then $X$ does not satisfy the Invariant Subspace Property. In other words, if for every $j_0\ge 1$, there exists $j\ge j_0$ such that $\ker p_{j+1}$ is of infinite codimension in $\ker p_{j}$, we have to be able to construct an operator on $X$ without non-trivial invariant subspaces. This will be done in Section~\ref{section3}.

An interesting application of Theorem~\ref{technique} concerns Fréchet spaces $X^{\mathbb{N}}$ where if $(X,(q_j))$ is a Fréchet space, we endow the space $X^{\mathbb{N}}$ with the seminorms $p_j((x_n)_n)=\max\{q_j(x_n):n\le j\}$. 

\begin{cor}\label{cor}
Let $X$ be a Fréchet space with a Schauder basis. The space $X^{\mathbb{N}}$ satisfies the Invariant Subspace Property if and only if $X$ is finite-dimensional or isomorphic to $\omega$.
\end{cor}

In particular, we deduce that the space $(\ell_2(\mathbb{N}))^{\mathbb{N}}$ does not satisfy the Invariant Subspace Property. 

\section{Properties of Schauder basis}

The goal of this paper consists in showing that if $(X,(p_j))$ is a Fréchet space with a Schauder basis and without continuous norm such that $\ker p_{j+1}$ is of infinite codimension in $\ker p_{j}$ for infinitely many $j$ then there exists an operator $T$ on $X$ such that each non-zero vector $x$ in $X$ is cyclic for $T$, i.e. $\text{span}\{T^nx:n\ge 0\}$ is dense in $X$ (so that $T$ admits no non-trivial invariant subspace). We will always assume in this paper that the sequence $(p_j)$ inducing the topology of $X$ is increasing. 

The construction of such an operator will rely on the ideas of Read and requires the existence of a Schauder basis. Therefore we first recall the following inequalities satisfied by Schauder basis.

\begin{theorem}[{\cite[Theorem 6 (p. 298)]{Jarchow}}]\label{Schauder}
Let $(X,(p_j))$ be a Fréchet space with a Schauder basis $(e_n)_{n\ge 0}$. Then for every $j\ge 1$, there exist $C_j>0$ and $J\ge 1$ such that for every $M\le N$, for every $x_0,\dots,x_N\in \K$, 
\[p_j\Big(\sum_{n=0}^{M}x_n e_n\Big)\le C_j p_J\Big(\sum_{n=0}^{N}x_n e_n\Big).\]
\end{theorem}

These inequalities will be mainly used to get that if $x=\sum_{n=0}^{\infty}x_n e_n$ then
\[p_j(x_n e_n)=p_j(\sum_{m=0}^nx_me_m-\sum_{m=0}^{n-1}x_me_m)\le p_j(\sum_{m=0}^nx_me_m)+p_j(\sum_{m=0}^{n-1}x_me_m)\le 2C_jp_J(x).\]

The values $p_j(e_n)$ are thus directly affected by the properties of kernels of seminorms $(p_j)$.

\begin{theorem}\label{Schauder2}
Let $(X,(p_j))$ be a Fréchet space with a Schauder basis $(e_n)_{n\ge 0}$. 
If for every $j_0$, there exists $j\ge j_0$ such that $\ker p_{j+1}$ is a subspace of infinite codimension in $\ker p_{j}$ then for every $j_0$, there exists $j>j_0$ such that
\[|\{n\ge 0:p_{j_0}(e_n)=0\ \text{and}\ p_{j}(e_n)>0\}|=\infty.\]
\end{theorem}
\begin{proof}
Let $j_0\ge 1$. It follows from Theorem~\ref{Schauder} that there exists $J>j_0$ and $C_{j_0}$ such that for every $n\ge 0$, every $x\in X$ with $x=\sum_{k=0}^{\infty}x_ke_k$, we have
\[p_{j_0}(x_n e_n)\le 2C_{j_0} p_J(x).\]
By assumption, there also exists $j\ge J$ such that $\ker p_{j+1}$ is of infinite codimension in $\ker p_j$. Therefore, for every $N\ge 0$, there exists $x^{(N)}=\sum_{n=N}^{\infty}x^{(N)}_n e_n\in X$ such that $p_{j}(x^{(N)})=0$ and $p_{j+1}(x^{(N)})>0$. Finally, since $p_{j+1}(x^{(N)})>0$, there exists $n\ge N$ such that $x^{(N)}_n\ne 0$ and $p_{j+1}(e_n)>0$, and we get
\[p_{j_0}(e_n)=\frac{p_{j_0}(x^{(N)}_ne_n)}{|x^{(N)}_n|}\le 2C_{j_0} \frac{p_J(x^{(N)})}{|x^{(N)}_n|}\le 2C_{j_0} \frac{p_j(x^{(N)})}{|x^{(N)}_n|}=0.\]
In other words, the set $\{n\ge 0:p_{j_0}(e_n)=0\ \text{and}\ p_{j+1}(e_n)>0\}$ is infinite.
\end{proof}

\section{Construction of an operator without invariant subspaces}\label{section3}

Let $(X,(p_j))$ be a Fréchet space with a Schauder basis $(e_n)_{n\ge 0}$ such that $\ker p_{j+1}$ is of infinite codimension in $\ker p_{j}$ for infinitely many $j$. In view of Theorems~\ref{Schauder} and \ref{Schauder2}, we can assume without loss of generality that $(p_j)$ is an increasing sequence of seminorms such that
for every $j\ge 1$, there exists $C_j\ge 1$ such that for every $M\le N$, every $x_0,\dots,x_N\in \K$, 
\begin{equation}\label{basis}
p_j\Big(\sum_{n=0}^{M}x_n e_n\Big)\le C_j p_{j+1}\Big(\sum_{n=0}^{N}x_n e_n\Big),
\end{equation}
and such that for every $j\ge 1$, 
\begin{equation}\label{ker}
|\{n\ge 0:p_{j}(e_n)=0\ \text{and}\ p_{j+1}(e_n)>0\}|=\infty.
\end{equation}
Moreover, we will assume that $(C_j)$ is an increasing sequence and for every $j\ge 1$, we will denote by $E_j$ the set $\{n\ge 0:p_{j}(e_n)=0\ \text{and}\ p_{j+1}(e_n)>0\}$ and by $E_0$ the set $\{n\ge 0:p_{1}(e_n)>0\}$.

Our construction of an operator without non-trivial invariant subspaces will be based on the construction of Read applied to multiples of the Schauder basis $(e_n)_{n\ge 0}$ which we will reorder in a suitable way. We will denote by $(u_n)_{n\ge 0}$ the sequence obtained by this reordering and we will then consider the map $T:\text{span}\{u_j:j\ge 0\}\to \text{span}\{u_j:j\ge 0\}$ defined by
\[T^{j}u_0=\left\lbrace
\begin{array}{ll}
u_j+T^{j-a_n}u_0 & \mbox{if $j\in [a_n,a_n+\Delta_n)$;}\\
u_j & \mbox{otherwise.}
\end{array}
\right.\]

The sequences  $(\Delta_n)_{n\ge 0}$ and $(a_n)_{n\ge 1}$ will be determined by induction such that $\Delta_0=0$, $\Delta_1=1$ and  $\Delta_{n+1}=a_n+\Delta_n$ for every $n\ge 1$ and we will assume that the sequence $(a_n)$ grows sufficiently rapidly to get
\begin{equation}\label{param}
2\Delta_{n+1}<a_{n+1}.
\end{equation}

For technical reasons, we will also need to consider a sequence $(N_n)_{n\ge 1}$ of positive integers such that for every $n\ge 1$, $N_n\le n$ and such that for every $i\in \mathbb{N}$, there exists an infinity of integers $n$ such that $N_n=i$.

The goal of this section consists in showing that for a convenient reordering $(u_n)$ of multiples of the Schauder basis $(e_n)$ and for a good choice of the sequence $(a_n)$, the operator $T$ can be extended on $X$ and does not possess non-trivial invariant subspaces. We will thus assume that $e_n=\lambda_{\sigma(n)} u_{\sigma(n)}$ where $\sigma$ is a bijection on $\mathbb{Z}_+$ and $\lambda_m\ne 0$ for every $m$. Moreover, we will select the elements $(u_j)$ so that for every $n\ge 2$, there exists $s_n\ge 0$ such that
\begin{equation}\label{tn}
\{\lambda_{0} u_0,\dots,\lambda_{\Delta_{n}-1} u_{\Delta_{n}-1}\}\supset\{e_0,\dots,e_{s_n}\}
\end{equation}
and  
\begin{equation}\label{tn2}
\{\lambda_{0} u_0,\dots,\lambda_{\Delta_{n}-1} u_{\Delta_{n}-1}\}\backslash\{e_0,\dots,e_{s_n}\}\subset \ker(p_{n+2}).
\end{equation}
Note that \eqref{tn} is equivalent to require that $[0,s_n]\subset \sigma^{-1}([0,\Delta_n))$ and that \eqref{tn2} means that $\sigma^{-1}([0,\Delta_n))\backslash [0,s_n]\subset \bigcup_{m\ge n+2}E_m$. Moreover, since $\sigma$ is a bijection, we have that $(s_n)$ tends to infinity. The construction of the reordering $(u_j)_{j\ge 1}$ will be divided along the intervals $[\Delta_n,\Delta_{n+1})$.

As usual, we note that $u_0$ is a cyclic vector for $T$ and that we can compute the vectors $Tu_j$ for every $j\ge 0$ from the iterates of $u_0$ under the action of $T$ as follows:

\[Tu_j=\left\lbrace
\begin{array}{ll}
u_{a_n}+u_0 & \mbox{if $j=a_n-1$;}\\
u_{a_n+\Delta_n}-u_{\Delta_n} & \mbox{if $j=a_n+\Delta_n-1$;}\\
u_{j+1} & \mbox{otherwise.}
\end{array}
\right.\]

The map $T$ can be extended continuously on $X$ under the following conditions.

\begin{lemma}[Continuity]\label{cont}
Assume that
\begin{enumerate}
\item for every $1\le l\le n$, every $j\in [\Delta_n,a_n)$, we have
\begin{equation}
p_{l+1}(u_{j})\ge 2^{j+1} p_{l}(u_{j+1}),
\label{finalcont1}
\end{equation}
and for every $1\le l\le n$, every $j\in [a_n,\Delta_{n+1}-1)$, we have
\begin{equation}
p_{l+1}(u_{j})\ge \Delta_n 2^{\Delta_n} p_{l}(u_{j+1}),
\label{finalcont1bis}
\end{equation}
\item for every $n\ge 1$, we have
\begin{equation}
p_{1}(u_{a_n-1})> 2^{a_n}p_{n}(u_0)
\label{finalcont2}
\end{equation}
\item for every $n\ge 1$, we have
\begin{equation}
p_{n}(u_{\Delta_n})=0,\quad p_{n+1}(u_{\Delta_n})>0 \quad \text{and} \quad p_{n+1}(u_{a_n+\Delta_n-1})>0.
\label{finalcont3}
\end{equation}
\item for every $j\ge 0$, every $l\ge 1$,
\begin{equation}
p_{l+1}(u_j)=0\quad \Rightarrow\quad p_{l}(u_{j+1})=0.
\label{finalcont4}
\end{equation}
\end{enumerate}
Then the map $T$ can be uniquely extended to a continuous operator on $X$ satisfying for every $N\ge 1$, every $x\in X$, 
\[p_N(Tx)\le 8C_{N+1}L_Np_{N+2}(x),\] where $C_{N+1}$ is given by \eqref{basis} and $L_N$ only depends on $\{u_0,\dots,u_{\Delta_{N+1}-1}\}$.
\end{lemma}
\begin{proof}
Let $N\ge 1$ and
\[L_N=\max\{2^jp_{N}(T u_{j})/p_{N+1}(u_j):p_{N+1}(u_{j})\ne 0,\ j<\Delta_{N+1}-1\}+1.\]
It follows that for every $j<\Delta_{N+1}-1$ such that $p_{N+1}(u_{j})\ne 0$, we get
\[p_N(Tu_j)\le 2^{-j}L_Np_{N+1}(u_{j}).\]
In particular, this is the case for $j=a_n-1$ with $n\le N$ in view of \eqref{finalcont2} and for $j=a_n+\Delta_n-1$ with $n<N$ in view of \eqref{finalcont3}.
Moreover, it follows from \eqref{finalcont4} that for every $j<\Delta_{N+1}-1$ such that $Tu_j=u_{j+1}$, we have
\[p_N(Tu_j)\le 2^{-j}L_Np_{N+1}(u_{j}).\]
We can thus conclude that for every $j<\Delta_{N+1}-1$,
\[p_N(Tu_j)\le 2^{-j}L_Np_{N+1}(u_{j}).\]

On the other hand, if $j\ge \Delta_{N+1}-1$ and $Tu_j=u_{j+1}$, it follows from \eqref{finalcont1} that if $j\in [\Delta_n,a_n-1)$ for some $n\ge N+1$ then 
\[p_N(Tu_j)\le 2^{-j}p_{N+1}(u_{j})\le 2^{-j}L_Np_{N+1}(u_{j})\]
and it follows from \eqref{finalcont1bis} that if $j\in [a_n,\Delta_{n+1}-1)$ for some $n\ge N+1$ then 
\[p_N(Tu_j)\le \Delta_{n}^{-1}2^{-\Delta_{n}}p_{N+1}(u_{j})\le \Delta_{n}^{-1}2^{-\Delta_{n}}L_Np_{N+1}(u_{j}).\]

Finally, if $j=a_n-1$ with $n\ge N+1$, it follows from \eqref{finalcont1} and \eqref{finalcont2} that
\begin{align*}
p_{N}(Tu_{a_n-1})=p_N(u_{a_n}+u_0)&\le 2^{-a_n}p_{N+1}(u_{a_n-1})+2^{-a_n}p_{1}(u_{a_n -1})\\
&\le 2^{-a_n+1} L_N p_{N+1}(u_{a_n-1})
\end{align*}
and if $j=a_n+\Delta_n-1$ with $n\ge N$, it follows from \eqref{finalcont3} that
\[p_{N}(Tu_{a_n+\Delta_n-1})=p_N(u_{\Delta_{n+1}}-u_{\Delta_n})=0.\]
Therefore, if we consider $x=\sum_{m=0}^{\infty}x_m e_m$, we can then let $Tx=\sum_{m=0}^{\infty}x_mTe_m$ and we have by \eqref{basis}
\begin{align*}
&p_N(Tx)\\
&\quad\le \sum_{m=0}^{\infty}p_N(T(x_m e_m))\\
&\quad\le L_N\left(\sum_{m=0}^{\infty}\frac{1}{2^{\sigma(m)}}p_{N+1}(x_m e_m)+\sum_{n=N+1}^{\infty}\sum_{m:\sigma(m)\in [a_n,\Delta_{n+1}-1)}\frac{1}{\Delta_n 2^{\Delta_n}}p_{N+1}(x_m e_m)\right)\\
&\quad\le 2L_N C_{N+1}\left(\sum_{m=0}^{\infty}\frac{1}{2^{\sigma(m)}}p_{N+2}(x)+\sum_{n=N+1}^{\infty}\sum_{m:\sigma(m)\in [a_n,\Delta_{n+1}-1)}\frac{1}{\Delta_n 2^{\Delta_n}}p_{N+2}(x)\right)\\
&\quad\le 8 L_N C_{N+1}p_{N+2}(x).
\end{align*}
\end{proof}

Before continuing the construction of our operator $T$, it is necessary to do some remarks concerning the above theorem. We first notice that for every $n\ge 1$, every $j\in [\Delta_n,a_n)$, if $\sigma^{-1}(j+1)\in E_n$ then \eqref{finalcont1} is satisfied. Such a choice will be done for every $j\in [\Delta_n,2\Delta_n)$ while we will choose $\sigma^{-1}(a_n-1)\in E_0$ so that \eqref{finalcont2} can be satisfied. This will not contradict \eqref{finalcont4}  since this last condition means that for every $j$, if $\sigma^{-1}(j)\in E_{m_j}$ and $\sigma^{-1}(j+1)\in E_{m_{j+1}}$ then $m_{j+1}\ge m_j-1$. The indexes $m_j$ can thus decrease from $n$ to $0$ on each block $[\Delta_n,a_n)$ if $a_n$ is sufficiently big. Moreover, \eqref{finalcont4} for $j=a_n+\Delta_n-1$ follows from \eqref{finalcont3} since $p_{n+1}(u_{a_n+\Delta_n-1})>0$ and $p_{n+1}(u_{a_n+\Delta_n})=0$. This fact will be important because we will fix the elements $u_j$ block by block along the intervals $[\Delta_n,\Delta_{n+1})$ and we want that the conditions to be satisfied by the elements in $[\Delta_n,\Delta_{n+1})$ do not depend from the elements $u_j$ for $j\ge \Delta_{n+1}$. It is also interesting to precise that we will fix the elements $u_j$ for $j\in [a_n,\Delta_{n+1})$ before knowing the exact value of $a_n$. This explains why the condition \eqref{finalcont1bis} is different to the condition \eqref{finalcont1}. Finally, in comparison to the case of Fréchet spaces with continuous norm, we needed to add \eqref{finalcont4} since we do not deal with norms. 


Given $P(T)=\sum_{n=0}^d \rho_n T^n$, we will denote by $|P|$ the sum $\sum_{n=0}^{d}|\rho_n|$. The following key lemma (see for instance \cite[Lemma 3]{Golinski2}) gives the existence of a finite family of polynomials $(P_l)_{l\le L}$ allowing to approach $u_0$ by $P_l(T)(y)$ for every $y$ in a suitable compact set. This will help us to show that every non-zero vector is cyclic under the action of $T$ since the vector $u_0$ is cyclic. In the case of Fréchet spaces with continuous norm~\cite{Menet}, the considered compact sets were given by  
\begin{equation*}
K_n=\Big\{y\in \text{span}(u_0,\dots,u_{t_{n}-1})~:~ p_{1}(y)\le 3/2\ \text{and}\  p_{1}(\tau_n y)\ge 1/2\Big\}
\end{equation*}
for some map $\tau$ and some increasing sequence $(t_n)$. Unfortunately, if $p_1$ is not a norm on $\text{span}(u_0,\dots,u_{t_{n}-1})$, these sets $K_n$ are not necessarily compact. This explains why in the paper~\cite{Menet}, we have restricted our study to Fréchet spaces with continuous norm. 

\begin{lemma}\label{P}
Let $\varepsilon>0$, let $a$ and $t$ be positive integers with $t>a$ and $(\gamma_0,\dots,\gamma_{t-1})$ be a perturbed canonical basis of $\text{\emph{span}}(u_0,\dots, u_{t-1})$ satisfying 
\[\gamma_0=u_0 \quad \text{and} \quad \gamma_a=\varepsilon u_a+u_0.\]
Let $\|\cdot\|$ be a norm on $span\{u_j:j\ge 0\}$ and $K\subset \text{\emph{span}}(u_0,\dots, u_{t-1})$ be a compact set in the induced topology such that $\nu:=a-\text{val}_{\gamma}(K)\ge 0$.\\
Then there is a number $D\ge 1$ satisfying
\[\sum_{j=0}^{t-1}|\lambda_j|\le D \quad\text{for every $y=\sum_{j=0}^{t-1}\lambda_j \gamma_j\in K$}\]
and a finite family of polynomials $(P_l)_{l=1}^L$ satisfying for every $1\le l\le L$,
\[ \text{\emph{val}}\ P_l\ge \nu,\quad \deg P_l< t \quad\text{and}\quad |P_l|\le D\]
such that for any $y \in K$ there is $1\le l\le L$ such that for each perturbed forward
shift $T:\text{\emph{span}}\{u_j:j\ge 0\}\to \text{\emph{span}}\{u_j:j\ge 0\}$ satisfying $T^ju_0=\gamma_j$ for every $1\le j\le t-1$, we have
\[\|P_l(T)y-u_0\|\le 2\varepsilon \|u_a\|+ D\max_{t\le j< 2t}\|T^ju_0\|.\]
\end{lemma}
\begin{remark}\text{\ }
We recall that by a perturbed canonical basis of $\text{span}(u_0,\dots, u_{t-1})$, we mean a family $(\gamma_0,\dots,\gamma_{t-1})$ satisfying for every $0\le j\le t-1$, $\gamma_j=\sum_{l=0}^j\mu_{j,l}u_l$ with $\mu_{j,j}\ne 0$ and by a forward shift $T:\text{span}\{u_j:j\ge 0\}\to \text{span}\{u_j:j\ge 0\}$, we mean an operator $T$ satisfying for every $j\ge 0$, $Tu_j=\sum_{l=0}^{j+1}\lambda_{j,l}u_l$ with $\lambda_{j,j+1}\ne 0$.
\end{remark}

Note that the norm $\|\cdot\|$ in the above lemma has only to be a norm on $\text{span}\{u_j:j\ge 0\}$ and not on $X$. This fact was successfully used by Golinski and Przestacki in their study of the space of smooth functions on the real line~\cite{Golinski3}. We adapt this idea and consider the sequence of compact sets $(K_n)_{n\ge 1}$ given by 
\begin{equation*}
K_n=\Big\{y\in \text{span}(u_0,\dots,u_{\Delta_{n+1}-1})~:~ \|y\|_1\le 3/2\ \text{and}\  \|\tau_n y\|_1\ge 1/2\Big\}
\end{equation*}
where, for every $x\in \text{span}\{u_j:j\ge 0\}=\text{span}\{e_n:n\ge 0\}$, if $x=\sum_{n=0}^{\infty}x_n e_n$, 
\begin{equation*}
\|x\|_1=p_{1}(x)+\sum_{j\ge 0}\frac{1}{C_{j+1}}\sum_{n\in E_j}\frac{1}{2^n} p_{j+1}(x_ne_n)
\end{equation*}
and where for every $n\ge 1$,
\begin{equation*}
\tau_n\Big(\sum_{j =0}^{\Delta_{n+1}-1}y_jT^ju_0\Big)=\sum_{j=0}^{a_{n}-1}y_jT^ju_0.
\end{equation*}
We remark that $\|\cdot\|_1$ is a well-defined norm on $\text{span}\{u_j:j\ge 0\}=\text{span}\{e_n:n\ge 0\}$ since $\bigcup_{j\ge 0}E_j=\mathbb{Z}_{+}$ and each set $K_n$ is thus compact. 

Lemma~\ref{P} will be applied to sets $K_n$ and to norms $\|\cdot\|_{N_n}$ where for every $N\ge 1$, the norm $\|\cdot\|_N$ is defined on $\text{span}\{e_n:n\ge 0\}$ by
\begin{equation}\label{normN}
\left\|x\right\|_N=p_{N}(x)+\sum_{j\ge 0}\frac{1}{C_{j+1}}\sum_{n\in E_j}\frac{1}{2^n} p_{j+1}(x_ne_n).
\end{equation} 
 This choice of norms allows us to get for every $x\in  \text{span}\{u_j:j\ge 0\}=\text{span}\{e_n:n\ge 0\}$, $
p_N(x)\le \|x\|_N$
but also to control $\|e_n\|_N$ by considering $n\in E_N$ with $n$ sufficiently big.

Since $a_n-\text{val}(K_n)\ge 1$ in the basis $\{u_0,\dots,T^{\Delta_{n+1}-1}u_0\}$ in view of the definition of $\tau_n$, we can apply Lemma~\ref{P} to $K_n$ in order to get the following corollary.

\begin{cor}\label{cor Pn}
For every $n\ge 1$, there exist $D_n\ge 1$ satisfying
\[\sum_{j=0}^{\Delta_{n+1}-1}|y_j|\le D_n \quad\text{for every $y=\sum_{j=0}^{\Delta_{n+1}-1}y_j T^ju_0\in K_n$}\]
and a family of polynomials $\mathcal{P}_n=(P_{n,k})_{k=1}^{k_n}$ satisfying 
\[\text{\emph{val}}(P_{n,k})\ge 1,\quad \deg(P_{n,k})<\Delta_{n+1} \quad\text{and}\quad |P_{n,k}|\le D_n\]
such that for each $y\in K_n$, there exists $1\le k\le k_n$ such that
\[\|P_{n,k}(T)y-u_0\|_{N_n}\le 2\|u_{a_n}\|_{N_n}+ D_n\max_{\Delta_{n+1}\le j< 2\Delta_{n+1}}\|u_j\|_{N_n}.\]
Moreover, $D_n$ and $\mathcal{P}_n$ only depend on $\{u_0,\dots,u_{\Delta_{n+1} -1}\}$.
\end{cor}

We recall that our goal consists in showing that under some conditions on the sequences $(a_n)$ and $(u_n)$, the operator $T$ has no non-trivial invariant subspace or, in other words, that every non-zero vector $x$ is cyclic under the action of $T$. Since, by definition of $T$, we already know that $u_0$ is cyclic, it will suffice to prove that for every $N\ge 1$, every $\varepsilon>0$, there exists a polynomial $Q$ such that $p_N(Q(T)x-u_0)<\varepsilon$. This polynomial $Q$ will be obtained thanks to Corollary~\ref{cor Pn}. In fact, we will show that every non-zero vector $x$ can be divided into two parts so that some multiple of the first part belongs to $K_n$ for some $n$ (Lemma~\ref{Kn}) and the second part does not perturb the desired estimation (Lemma~\ref{tail}).

\begin{lemma}[Sets $K_n$]\label{Kn}
If for every $n\ge 1$, every $j\in [a_{n},a_{n}+\Delta_{n})$, 
\begin{equation}
\label{K1}
\sigma^{-1}(j)\in E_{N_n} \quad\text{and}\quad p_{N_{n}+1}(u_j)\ge 2^{\Delta_n}\|T^{j-a_n}u_0\|_1,
\end{equation}
 then for every $x\in X\backslash\{0\}$, there exists $M>0$ such that for every integer $N\ge 1$, every increasing sequence $(n_k)$ such that $N_{n_k}=N$, we have that for all but finitely many $k$, there exists $M_k\ge M$ such that
\[\frac{\pi^{(u)}_{[0,\Delta_{n_k+1})}x}{M_k}\in K_{n_k}\]
where $\pi^{(u)}_{[0,\Delta_{n_k+1})}x=\sum_{l\in \sigma^{-1}([0,\Delta_{n_k+1}))}x_l e_l$ if $x=\sum_{l=0}^{\infty}x_l e_l$. 
\end{lemma}
\begin{proof}
We recall that \begin{equation*}
K_n=\Big\{y\in \text{span}(u_0,\dots,u_{\Delta_{n+1}-1})~:~ \|y\|_1\le 3/2\ \text{and}\  \|\tau_n y\|_1\ge 1/2\Big\}
\end{equation*}
and that
\begin{equation*}
\tau_n\Big(\sum_{j =0}^{\Delta_{n+1}-1}y_jT^ju_0\Big)=\sum_{j=0}^{a_{n}-1}y_jT^ju_0.
\end{equation*}
Let $n\ge 1$. We start by computing $\|\tau_n u_j\|_1$ for every $j\in [0,\Delta_{n+1})$.
\begin{itemize}
\item If $j< a_{n}$ then $\tau_n u_j=u_j$. 
\item If $j\in [a_n,a_n+\Delta_{n})$ then $\tau_n u_j=-T^{j-a_n}u_0$ and thus by \eqref{K1}
\[\|\tau_n u_j\|_1=\|T^{j-a_n}u_0\|_1\le \frac{1}{2^{\Delta_n}} p_{N_{n}+1}(u_j).\]
\end{itemize}
Let $y=\sum_{j=0}^{\Delta_{n+1}-1}y_ju_j$. We deduce that 
\begin{align*}
\left\|\tau_n \left(\sum_{j=a_n}^{\Delta_{n+1}-1}y_ju_j\right)\right\|_1&\le \sum_{j=a_n}^{a_n+\Delta_n-1}\|y_j \tau_n u_j\|_1\le \sum_{j=a_n}^{a_n+\Delta_n-1}\frac{1}{2^{\Delta_n}}p_{N_{n}+1}(y_j u_j)\\
&\le \sum_{j=a_n}^{a_n+\Delta_n-1}\frac{2C_{N_{n}+1}}{2^{\Delta_n}}p_{N_{n}+2}(y)\le \frac{2\Delta_n C_{N_{n}+1}}{2^{\Delta_n}}p_{N_{n}+2}(y)
\end{align*}
where $C_{N_{n}+1}$ is given by \eqref{basis}.
On the other hand, since $N_n\ge 1$ and $\sigma^{-1}(j)\in E_{N_{n}}$ for every $j\in [a_n,a_n+\Delta_n)$,  it follows from \eqref{tn} that 
\begin{align*}
\left\|\sum_{j=a_n}^{\Delta_{n+1}-1}y_ju_j\right\|_1&= \frac{1}{C_{N_{n}+1}}\sum_{j=a_n}^{a_n+\Delta_{n}-1}\frac{1}{2^{\sigma^{-1}(j)}} p_{N_{n}+1}(y_{j}u_{j})\\
&\le  \frac{1}{C_{N_{n}+1}}\sum_{j=a_n}^{a_n+\Delta_n-1} \frac{2C_{N_n+1}}{2^{\sigma^{-1}(j)}} p_{N_n+2}(y)\\
&\le  \sum_{k\ge \min(\sigma^{-1}[a_n,a_n+\Delta_n))}\frac{1}{2^{k-1}}p_{N_n+2}(y)\\
&\le \frac{1}{2^{s_{n}-1}}p_{N_n+2}(y).
\end{align*}
Let $x\in X\backslash\{0\}$ with $x=\sum_{l=0}^{\infty}x_l e_l$. We let $\pi^{(e)}_{[0,s]}(x)=\sum_{l=0}^sx_le_l$ for every $s\ge 0$. Let $M=\|\pi^{(u)}_{[0,\Delta_{m_0 +1})}x\|_{1}-p_1(\pi^{(u)}_{[0,\Delta_{m_0 +1})}x)$ where the index $m_0$ is chosen so that $M>0$. Let $N\ge 1$ and let $(n_k)$ be an increasing sequence such that $N_{n_k}=N$. We let $M_k=\|\pi^{(u)}_{[0,\Delta_{n_k +1})}x\|_{1}$ and we remark that $M_k\ge M$ when $n_k\ge m_0$. If $\pi^{(u)}_{[0,\Delta_{n_k +1})}x=\sum_{j=0}^{\Delta_{n_k +1}-1}y_ju_j$, it then follows from \eqref{tn} and \eqref{tn2} that
\begin{align*}
\|\tau_{n_k}\pi^{(u)}_{[0,\Delta_{n_k +1})}x\|_1&=\left\|\tau_{n_k}\left(\sum_{j=0}^{a_{n_k}-1}y_ju_j\right)+\tau_{n_k}\left(\sum_{j=a_{n_k}}^{\Delta_{n_k +1}-1}y_ju_j\right)\right\|_1\\
&\ge \left\|\sum_{j=0}^{a_{n_k}-1}y_ju_j\right\|_1-\left\|\tau_{n_k}\left(\sum_{j=a_{n_k}}^{\Delta_{n_k +1}-1}y_ju_j\right)\right\|_1\\
&\ge \|\pi^{(u)}_{[0,\Delta_{n_k +1})}x\|_1-\left\|\sum_{j=a_{n_k}}^{\Delta_{n_k +1}-1}y_ju_j\right\|_1-\frac{2\Delta_{n_k}C_{N+1}}{2^{\Delta_{n_k}}}p_{N+2}(\pi^{(u)}_{[0,\Delta_{n_k +1})}x)\\
& \ge M_k-\frac{1}{2^{s_{n_k}-1}}p_{N+2}(\pi^{(u)}_{[0,\Delta_{n_k +1})}x)-\frac{2\Delta_{n_k}C_{N+1}}{2^{\Delta_{n_k}}}p_{N+2}(\pi^{(u)}_{[0,\Delta_{n_k +1})}x)\\
&=M_k-\frac{1}{2^{s_{n_k}-1}}p_{N+2}(\pi^{(e)}_{[0,s_{n_k +1}]}x)-\frac{2\Delta_{n_k}C_{N+1}}{2^{\Delta_{n_k}}}p_{N+2}(\pi^{(e)}_{[0,s_{n_k +1}]}x)
\end{align*}
since $N=N_{n_k}\le n_k$ by definition of $(N_n)$. Finally, since for every $m$, the seqence $(p_m(\pi^{(e)}_{[0,s_{n_k +1}]}x))_k$ is bounded, we conclude that for all but finitely many $k$, $M_k\ge M$ and $\|\tau_{n_k}\pi^{(u)}_{[0,\Delta_{n_k +1})}x\|_1\ge \frac{M_k}{2}$, and therefore that
$\frac{\pi^{(u)}_{[0,\Delta_{n_k +1})}x}{M_k}\in K_{n_k}$.
\end{proof}

We recall that at the beginning of this section, we have assumed that the sequence $(C_j)$ given by \eqref{basis} is increasing. This assumption was done in order to simplify the statement of the following lemma. For the same reason, we will also assume that the sequence $(D_n)_{n\ge 1}$ given by Corollary~\ref{cor Pn} is increasing.
\begin{lemma}[Tails]\label{tail}
Assume that for every $n\ge 2$, every $j\in [\Delta_n,\Delta_{n+1})$,
\begin{enumerate}
\item for every $l\le n$, every $1\le r< \Delta_{n}$, if $j\in [\Delta_n, a_n)$ then
\begin{equation}
\label{tail1}
p_{l+1}(u_{j})\ge 2^{j+1}C_{n+1}D_{n-1}p_l(u_{j+r});
\end{equation}
and if $j\in [a_n,\Delta_{n+1})$ and $j+r<\Delta_{n+1}$ then
\begin{equation}
\label{tail1bis}
p_{l+1}(u_{j})\ge \Delta_n 2^{\Delta_n+1}C_{n+1}D_{n-1}p_l(u_{j+r});
\end{equation}
\item if $j\in [\Delta_{n},2\Delta_{n})$ then 
\begin{equation}
\label{tail2}
p_{n}(u_j)=0;
\end{equation} 
\item if $j\in [a_{n}-\Delta_{n},a_{n})$ then
\begin{equation}
\label{tail3}
p_1(u_j)\ge 2^{j+1}C_{n+1}D_{n-1}\sup_{m< \Delta_n}p_n(T^m u_0).
\end{equation} 
\end{enumerate}
Then for every $n\ge 1$, every polynomial $P$ with $\text{\emph{val}}(P)\ge 1$, $\deg(P)< \Delta_{n+1}$ and $|P|\le D_{n}$, every $x\in \overline{\text{\emph{span}}}\{e_j:\sigma(j)\ge \Delta_{n+1}\}$, we have
\[p_{N_n}(P(T)x)\le 4p_{N_n+2}(x).\]  
\end{lemma}
\begin{proof}
Let $n\ge 2$ and $j\in [\Delta_n,\Delta_{n+1})$. Let $2\le m\le n$ and $1\le r< \Delta_m$. If $j\in [\Delta_n,a_n-r)$ then it follows from \eqref{tail1} that
\[p_{N_{m-1}}(T^ru_j)=p_{N_{m-1}}(u_{j+r})\le \frac{1}{2^{j+1}C_{N_{m-1}+1}D_{m-1}}p_{N_{m-1} +1}(u_j).\]
If $j\in [a_{n}-r,a_{n})$ then it follows from \eqref{tail1} and \eqref{tail3} that
\begin{align*}
p_{N_{m-1}}(T^ru_j)&=p_{N_{m-1}}(u_{j+r}+T^{j+r-a_{n}}u_0)\\
&\le  \frac{1}{2^{j+1}C_{N_{m-1}+1}D_{m-1}}p_{N_{m-1} +1}(u_j)+
\frac{1}{2^{j+1}C_{N_{m-1}+1}D_{m-1}}p_{1}(u_j)\\
&\le  \frac{1}{2^jC_{N_{m-1}+1}D_{m-1}}p_{N_{m-1}+1}(u_j)
\end{align*}
since $N_{m-1}\le m-1$ and $(C_j)_j$ and $(D_n)_n$ are increasing.
On the other hand, if $j\in [a_{n},a_{n}+\Delta_n-r)$ then it follows from \eqref{tail1bis} that
\[p_{N_{m-1}}(T^ru_j)=p_{N_{m-1}}(u_{j+r})\le \frac{1}{\Delta_n2^{\Delta_n+1}C_{N_{m-1}+1}D_{m-1}}p_{N_{m-1} +1}(u_j).\]
Finally, if $j\in [a_n+\Delta_n-r,a_n+\Delta_n)$ then  $j+r\in [a_n+\Delta_n,a_{n+1})$ since $a_{n}+2\Delta_n<2\Delta_{n+1}<a_{n+1}$ by \eqref{param} and it follows from \eqref{tail2} that
\begin{align*}
p_{N_{m-1}}(T^ru_j)=p_{N_{m-1}}(u_{j+r}-u_{j+r-a_n})=0
\end{align*}
since $j+r\in [\Delta_{n+1},2\Delta_{n+1})$ and $j+r-a_n\in [\Delta_{n},2\Delta_{n})$.\\

Let $n\ge 1$ and let $P$ be a polynomial with $\text{val}(P)\ge 1$, $\deg(P)< \Delta_{n+1}$ and $|P|\le D_{n}$. It follows from the previous inequalities that if $x=\sum_{j:\sigma(j)\ge \Delta_{n+1}}x_je_j$ then
\begin{align*}
p_{N_n}(P(T)x)&\le |P|\max_{1\le r< \Delta_{n+1}}p_{N_n}(T^rx)\\
&\le D_n\sum_{j:\sigma(j)\ge \Delta_{n+1}}|x_j|\max_{1\le r< \Delta_{n+1}}p_{N_n}(T^re_j)\\
&\le D_n\left(\sum_{j:\sigma(j)\ge \Delta_{n+1}}\frac{1}{2^{\sigma(j)}C_{N_n+1}D_{n}}|x_j|p_{N_n+1}(e_j)\right.\\
&\quad\quad\left.
+\sum_{m\ge n+1}\sum_{j:\sigma(j)\in [a_{m},a_{m}+\Delta_m)}\frac{1}{\Delta_m 2^{\Delta_m+1}C_{N_n+1}D_{n}}|x_j|p_{N_n+1}(e_j)\right)\\
&\le \left(\sum_{j:\sigma(j)\ge \Delta_{n+1}}\frac{1}{2^{\sigma(j)-1}}+ \sum_{m\ge n+1}\frac{1}{2^{\Delta_m}}\right) p_{N_n+2}(x)\le 4p_{N_n+2}(x).
\end{align*}
\end{proof}
\begin{remark}
In the same way as for condition \eqref{finalcont1bis} in Lemma~\ref{cont}, the reason to consider condition \eqref{tail1bis} relies on the fact that the value of $a_n$ will not be known when we will fix the elements $u_j$ for $j\in [a_n,\Delta_{n+1})$.
\end{remark}

Thanks to all these lemmas, we are now able to write a list of conditions such that if the sequences $(a_n)_{n\ge 1}$ and $(u_n)_{n\ge 0}$ satisfy each of these conditions then $T$ admits no non-trivial invariant subspace.

\begin{lemma}[Final result]\label{finalresult}
Under the assumptions \eqref{tn}-\eqref{tn2}, \eqref{finalcont1}-\eqref{finalcont4}, \eqref{K1} and \eqref{tail1}-\eqref{tail3}, if for every $n\ge 1$, we have
\begin{equation}
\label{final1}
\|u_{a_n}\|_{N_n}\le \frac{1}{2^n}
\end{equation}
and for every  $\Delta_{n+1}\le j< 2\Delta_{n+1}$, we have
\begin{equation}
\label{final2}
\|u_j\|_{N_{n}}\le \frac{1}{2^{n}D_{n}}
\end{equation}
 then every non-zero vector in $X$ is cyclic for $T$ and thus $T$ has no non-trivial invariant subspace.
\end{lemma}
\begin{proof}
Let $x\in X\backslash\{0\}$. Let $N\ge 1$. Since $u_0$ is a cyclic vector for $T$, it suffices to show that there exists a sequence of polynomial $(Q_k)_k$ such that $p_N(Q_k(T)x-u_0)\to 0$ in order to deduce that $x$ is cyclic for $T$.

Let $(n_k)$ be an increasing sequence such that $N_{n_k}=N$. By Lemma~\ref{Kn}, there exists $M>0$ such that for all but finitely many $k$, there exists $M_k\ge M$ such that
\[y:=\frac{\pi^{(u)}_{[0,\Delta _{n_k +1})}x}{M_k}\in K_{n_k}.\]
By Corollary~\ref{cor Pn}, there then exists a polynomial $P_{n_k,l}$ with $\text{val}(P_{n_k,l})\ge 1$, $\deg(P_{n_k,l})< \Delta_{n_k +1}$ and $|P_{n_k,l}|\le D_{n_k}$ such that 
\begin{align*}
p_{N}(P_{n_k,l}(T)y-u_0)&\le \|P_{n_k,l}(T)y-u_0\|_{N}\\
&\le 2\|u_{a_{n_k}}\|_{N}+D_{n_k}\max_{\Delta_{n_k+1}\le j< 2\Delta_{n_k +1}}\|u_j\|_{N}\\
&\le \frac{3}{2^{n_k}}\quad \text{by \eqref{final1} and \eqref{final2}}.
\end{align*}
On the other hand, it follows from Lemma~\ref{tail} that
\[p_{N}(P_{n_k,l}(T)\pi^{(u)}_{[\Delta_{n_k +1},+\infty)}x)\le 4 p_{N+2}(\pi^{(u)}_{[\Delta_{n_k +1},+\infty)}x)\]
where $\pi^{(u)}_{[\Delta_{n_k +1},+\infty)}x$ is given by $x-\pi^{(u)}_{[0,\Delta_{n_k +1})}x$.
Since $N_n\le n$ for every $n\ge 1$, we conclude thanks to \eqref{tn2} that
\begin{align*}
p_N(\frac{1}{M_k}P_{n_k,l}(T)x-u_0)&\le p_N(P_{n_k,l}(T)y-u_0)+ p_N(\frac{1}{M_k}P_{n_k,l}(T)(\pi^{(u)}_{[\Delta_{n_k +1},+\infty)}x))\\
&\le \frac{3}{2^{n_k}}+\frac{4}{M}p_{N+2}(\pi^{(u)}_{[\Delta_{n_k +1},+\infty)}x)\\
&=\frac{3}{2^{n_k}}+\frac{4}{M}p_{N+2}(\pi^{(e)}_{(s_{n_k +1},+\infty)}x)\to 0,
\end{align*}
where $\pi^{(e)}_{(s_{n_k +1},+\infty)}x=\sum_{l=s_{n_k +1}+1}^{+\infty}x_le_l$ if $x=\sum_{l=0}^{\infty}x_le_l$.
Each non-zero vector in $X$ is thus cyclic for $T$ and therefore $T$ does not possess any non-trivial invariant subspace.
\end{proof}

We are now able to conclude the proof of Theorem~\ref{technique}.

\begin{proof}[Proof of Theorem~\ref{technique}]
We first recall that it is already known that if there exists $j_0\ge 1$ such that $\ker p_{j+1}$ is a subspace of finite codimension in $\ker p_j$ for every $j\ge j_0$ then $X$ satisfies the Invariant Subspace Property~(\cite[Theorem 2.1]{Menet}).

On the other hand, if we assume that $(X,(p_j))$ is a Fréchet space with a Schauder basis $(e_n)_{n\ge 0}$  such that $\ker p_{j+1}$ is of infinite codimension in $\ker p_{j}$ for infinitely many $j$, then we can assume without loss of generality that \eqref{basis} and \eqref{ker} are satisfied and in view of Lemma~\ref{finalresult}, it remains to prove that it is possible to construct a sequence $(a_n)_{n\ge 1}$ and a sequence $(u_j)_{j\ge 0}$ satisfying \eqref{tn}-\eqref{tn2}, \eqref{finalcont1}-\eqref{finalcont4}, \eqref{K1}, \eqref{tail1}-\eqref{tail3} and \eqref{final1}-\eqref{final2}. 
 
To this end, we first let $u_0=e_{i_0}$ and $u_1=e_{i_1}$ with $i_0\ne i_1$ and $i_0, i_1\in E_1$ so that \eqref{finalcont3} is satisfied for $n=1$ and \eqref{finalcont4} is satisfied for $j=0$. We recall that $\Delta_0=0$ and that $\Delta_1=1$. We then select an index $i_{a_1}\notin\{i_0,i_1\}$ with $i_{a_1}\in E_{N_1}=E_1$ sufficiently big so that by letting $u_{a_1}=\frac{2^{\Delta_1}\|u_0\|_{1}}{p_{N_1+1}(e_{i_{a_1}})}e_{i_{a_1}}$, we get \eqref{finalcont3}, \eqref{K1} and \eqref{final1}, i.e.
\[p_2(u_{a_1})>0,\quad p_{N_1+1}(u_{a_1})\ge 2^{\Delta_1}\|u_0\|_1\quad\text{and}\quad \|u_{a_1}\|_{N_1}=\|u_{a_1}\|_{1}=\frac{2^{\Delta_1}\|u_0\|_1}{2^{i_{a_1}}C_2}\le \frac{1}{2}.\]
Note that the choice of the index $i_{a_1}$ is independent of the value of $a_1$.
We now select $4$ distinct indices $\mathcal{K}=\{k_0,k_{1},k_2,k_{3}\}$ different to $i_0$, $i_1$ and $i_{a_1}$ with $k_l\in E_l$. Let $s_2=\max\{i_0,i_1,i_{a_1},k_0,k_1,k_2,k_3\}$. We can complete the family $\mathcal{K}$ so that
\begin{itemize}
\item $\mathcal{K}\cap\{i_0,i_1,i_{a_1}\}=\emptyset$
\item $\mathcal{K}\cup\{i_0,i_1,i_{a_1}\}\supset [0,s_2]$,
\item if $j\in \mathcal{K}$ and $j>s_2$ then $j\in E_l$ for some $l\ge 4$,
\item for every $l\ge 1$, if $\mathcal{K}\cap E_l\ne \emptyset$ then $\mathcal{K}\cap E_m\ne \emptyset$ for every $m< l$. 
\end{itemize}
We then let $a_1=|\mathcal{K}|+2$ and for each $n\in(1,a_1)$, we take for $u_n$ a non-zero multiple of $e_s$ with $s\in \mathcal{K}$ so that \eqref{finalcont1}, \eqref{finalcont2} and \eqref{finalcont4} are satisfied. More precisely, we consider an enumeration $(i_j)_{1<j<a_1}$ of $\mathcal{K}$ such that if $i_j\in E_s$ then $i_{j+1}\in E_s\cup E_{s-1}$ and we let  $u_j=\alpha_j e_{i_j}$. We can then deduce that \eqref{finalcont4} is satisfied for every $j< a_1$ and we can choose $\alpha_{j}$ sufficiently big to get \eqref{finalcont1} for every $\Delta_1\le j\le a_1-1$ and to get \eqref{finalcont2} for $n=1$ since $i_{a_1-1}\in E_0$. We conclude that the family $\{u_0,\dots,u_{\Delta_2 -1}\}$ satisfies \eqref{tn}-\eqref{tn2}, \eqref{finalcont1}-\eqref{finalcont4}, \eqref{K1} and  \eqref{final1}. Note that conditions \eqref{tail1}-\eqref{tail3} and \eqref{final2} only concern the indices bigger than $\Delta_2$.

We continue the construction of $(u_k)_{k\ge 0}$ by induction along the intervals $[\Delta_n,\Delta_{n+1})$. We thus assume that  $u_0,...u_{\Delta_{n}-1}$ have been chosen for some $n\ge 2$ so that \eqref{tn}- \eqref{tn2}, \eqref{finalcont1}-\eqref{finalcont4}, \eqref{K1}, \eqref{tail1}-\eqref{tail3} and \eqref{final1}-\eqref{final2} are satisfied. We will divide the selection of elements $u_j$ for $j\in [\Delta_n,\Delta_{n+1})$ into three steps: $[\Delta_n,2\Delta_n)$, $[a_{n},a_{n}+\Delta_n)$ and $[2\Delta_n,a_n)$.

Let $D_{n-1}$ be the real number given by Corollary~\ref{cor Pn}. We first select $\Delta_n$ distinct indices $(i_j)_{j\in [\Delta_n,2\Delta_n)}$ different from the indices already selected such that $i_j\in E_{n}$ for every $j\in [\Delta_n,2\Delta_n)$ and we let $u_j=\alpha_j e_{i_{j}}$ with $\alpha_j\ne 0$.
It follows that \eqref{finalcont1} and \eqref{finalcont4} are satisfied for every $j\in [\Delta_n,2\Delta_n-1)$, that \eqref{tail2} and the first two conditions of \eqref{finalcont3} are satisfied and that \eqref{tail1} is satisfied for $j+r<2\Delta_n$. Moroever, we can choose $\alpha_j$ sufficiently small so that $\eqref{final2}$ is satisfied.

We complete by choosing the elements $u_j$ for $j\in [a_{n},a_{n}+\Delta_n)$. We first select $\Delta_{n}-1$ distinct indices $(i_j)_{j\in (a_{n},a_{n}+\Delta_n) }$ different from the indices already selected such that $i_j\in E_{N_n}$ and we let $u_j=\alpha_j e_{i_{j}}$ for every $j\in (a_{n},a_n+\Delta_n)$. It follows that the last condition of \eqref{finalcont3} is satisfied and that \eqref{finalcont4} is satisfied for every $j\in (a_n,a_n+\Delta_n-1)$.  Moreover, we can choose $\alpha_j$ sufficiently big so that \eqref{K1} is satisfied, that \eqref{finalcont1bis} is satisfied for $j\in (a_n,\Delta_{n+1}-1)$ and that \eqref{tail1bis} is satisfied for $j\in (a_n,\Delta_{n+1})$ (even if the value of $a_n$ has not still be determined). We then select a sufficiently big index $i_{a_n}\in E_{N_n}$  such that by letting 
\[u_{a_{n}}=\big(\Delta_n2^{\Delta_n+1}C_{n+1}D_{n-1}\max\{p_{n}(u_j):j\in (a_n,\Delta_{n+1})\}+2^{\Delta_{n}}\|u_0\|_1\big)\frac{e_{i_{a_n}}}{p_{N_n+1}(e_{i_{a_n}})},\] we get \eqref{finalcont1bis}, \eqref{K1}, \eqref{tail1bis} and \eqref{final1}, i.e. for every $1\le l\le n$, every $j\in (a_n,\Delta_{n+1})$,
\[p_{l+1}(u_{a_n})\ge \Delta_n 2^{\Delta_n}p_{l}(u_{a_n+1}),\quad p_{N_n+1}(u_{a_n})\ge 2^{\Delta_n}\|u_0\|_1,\]
\[p_{l+1}(u_{a_n})\ge \Delta_n 2^{\Delta_n+1}C_{n+1}D_{n-1}p_{l}(u_j)\]
and
\[\|u_{a_{n}}\|_{N_n}=\frac{\big(\Delta_n2^{\Delta_n+1}C_{n+1}D_{n-1}\max\{p_{n}(u_j):j\in (a_n,\Delta_{n+1})\}+2^{\Delta_{n}}\|u_0\|_1\big)}{2^{i_{a_n}}C_{N_n+1}}\le \frac{1}{2^n}.\]
Moreover, \eqref{finalcont4} is satisfied for $j=a_n$ since $i_{a_n}$ and $i_{a_{n}+1}$ belong to $E_{N_n}$. 

We are now looking for the elements $u_j$ for $j\in [2\Delta_{n},a_n)$. Let $I_n$ be the set of already selected indices.  We select $(n+3)\Delta_{n}$ distinct indices $\mathcal{K}=\{k_{0},...,k_{(n+3)\Delta_n-1}\}$ such that $\mathcal{K}\cap I_n=\emptyset$ and such that $k_{l\Delta_n+r}\in E_l$ for every $0\le l\le n+2$ and $0\le r<\Delta_n$. Let $s_{n+1}$ be the maximum of $\mathcal{K}\cup I_n$. We can now complete the family $\mathcal{K}$ so that
\begin{itemize}
\item $\mathcal{K}\cap I_n=\emptyset$,
\item $\mathcal{K}\cup I_n\supset [0,s_{n+1}]$,
\item if $j\in \mathcal{K}$ and $j>s_{n+1}$ then $j\in E_l$ for some $l\ge n+3$,
\item for every $l\ge 1$, if $\mathcal{K}\cap E_l\ne \emptyset$ then $\mathcal{K}\cap E_m\ne \emptyset$ for every $m< l$.
\end{itemize} 
We then let $a_{n}=|\mathcal{K}|+2\Delta_n$ and for each $j\in[2\Delta_n,a_n)$, we take for $u_j$ a multiple of $e_s$ where $s\in \mathcal{K}$ so that \eqref{finalcont1}, \eqref{finalcont2}, \eqref{finalcont4}, \eqref{tail1} and \eqref{tail3} are satisfied. More precisely, we consider an enumeration $(i_j)_{2\Delta_n\le j<a_n}$ of $\mathcal{K}$ such that if $i_j\in E_s$ then $i_{j+1}\in E_s\cup E_{s-1}$ and we let  $u_j=\alpha_j e_{i_j}$. 
Since for every $s\le n+2$, the set $\mathcal{K}$ contains at least $\Delta_n$ indices in $E_s$, we can deduce that  $i_j\in E_0$ for every $j\in [a_n-\Delta_n,a_n)$ and that
 $i_j\in \bigcup_{s\ge n+2}E_s$ for every $j\in [2\Delta_n,3\Delta_n)$. It follows that for any choice of $(\alpha_i)_{i\in [2\Delta_n,a_n)}$, \eqref{finalcont4} is satisfied for every $j<\Delta_{n+1}$, \eqref{finalcont1} is satisfied for $j=2\Delta_{n}-1$ and \eqref{tail1} is satisfied for every $j\in [\Delta_n,2\Delta_n)$. We can then choose $\alpha_{i}$ sufficiently big for every $i\in [2\Delta_n,a_n)$ so that we get \eqref{finalcont2} and \eqref{tail3} but also \eqref{finalcont1} and \eqref{tail1} for every $j\in [\Delta_n, a_n)$. We conclude that the family $\{u_0,\dots,u_{\Delta_n+1 -1}\}$ satisfies \eqref{tn}-\eqref{tn2}, \eqref{finalcont1}-\eqref{finalcont4}, \eqref{K1}, \eqref{tail1}-\eqref{tail3} and  \eqref{final1}-\eqref{final2}.

It follows from Lemma~\ref{finalresult} that $X$ does not satisfy the Invariant Subspace Property.
\end{proof}

We end this paper by showing how Theorem~\ref{technique} leads to the characterization of spaces $X^{\mathbb{N}}$ satisfying the Invariant Subspace Property.

\begin{proof}[Proof of Corollary~\ref{cor}]
Let $(X,(q_j)_j)$ be a Fréchet space with a Schauder basis where we assume that $(q_j)$ is an increasing sequence. It follows that $X^{\mathbb{N}}$ endowed with the seminorms $p_j((x_n)_n)=\max\{q_j(x_n):n\le j\}$ is a Fréchet space with a Schauder basis and without continuous norm since none of the seminorms $p_j$ is a norm. Moreover, the sequence $(p_j)$ is increasing and we can thus apply Theorem~\ref{technique} to the space $(X^{\mathbb{N}},(p_j))$.

If $X$ is finite-dimensional or if $X$ is isomorphic to $\omega$ then for every $j\ge 1$, $\ker p_{j+1}$ is of finite codimension in $X$ and thus in particular in $\ker p_j$ and it follows from Theorem~\ref{technique} that $X^{\mathbb{N}}$ satisfies the Invariant Subspace Property. On the other hand, if $X$ is infinite dimensional but is not isomorphic to $\omega$ then there exists $j_0\ge 1$ such that $\ker q_{j_0}$ is of infinite codimension in $X$ and it follows that for every $j\ge j_0$, $\ker p_{j+1}$ is of infinite codimension in $\ker p_j$. We can then conclude by applying Theorem~\ref{technique}.
\end{proof}

\begin{remark}
Corollary~\ref{cor} means that if $X$ is a Fréchet space with a Schauder basis then $X^{\mathbb{N}}$ satisfies the Invariant Subspace Property if and only if $X^{\mathbb{N}}$ is isomorphic to $\omega$.
\end{remark}

\end{document}